\documentclass[12pt,reqno]{article}
mm \textheight=8.9in

\usepackage{amssymb}
\usepackage{amsmath}
\usepackage{amsthm}
\usepackage{color}

\newcommand{\br}[3]{{$#1$}$\lower4pt\hbox{$\tp\atop\raise4pt \hbox{$\scriptscriptstyle{#2}$}$} ${$#3$}}
\newcommand{\tw}[3]{{$#1$}${\,\scriptscriptstyle {#2}}\atop\raise9pt\hbox{$\scriptstyle\tp$} ${$#3$}}
\newcommand{\ttps}[2]{{#1}\raise5pt\hbox{$\lower12pt\hbox{$\scriptstyle\tp$}\atop \lower0pt\hbox{$\tilde\;$}$}\raise4.5pt\hbox{${\scriptstyle{#2}}$}}
\newcommand{\st}[1]{\mbox{${\,\scriptscriptstyle {#1}}\atop\raise5.5pt\hbox{$*$}$}}

\newcommand{\rd}[1]{\mbox{${\,\scriptscriptstyle {#1}}\atop\raise5.5pt\hbox{$\bullet$}$}}
\newcommand{\rt}[1]{\otimes_\chi}
\newcommand{\lt}[1]{\mbox{${\,\scriptscriptstyle {#1}}\atop\raise5.5pt\hbox{$\ltimes$}$}}
\newcommand{\btr}{\raise1.2pt\hbox{$\scriptstyle\blacktriangleright$}\hspace{2pt}}
\newcommand{\btl}{\raise1.2pt\hbox{$\scriptstyle\blacktriangleleft$}\hspace{2pt}}

\newcommand{\lcr}{\raise1.0pt \hbox{${\scriptstyle\rightharpoonup}$}}
\newcommand{\rcr}{\raise1.0pt \hbox{${\scriptstyle\leftharpoonup}$}}

\newcommand{\ttp}{{\lower12pt\hbox{$\tp$}\atop \hbox{$\tilde\;$}}}

\newcommand{\id}{\mathrm{id}}

\newcommand{\Tc}{\mathcal{T}}

\newcommand{\Bc}{\mathcal{B}}

\newcommand{\Ac}{{\mathcal{A}}}

\newcommand{\Ru}{\mathcal{R}}

\newcommand{\Fc}{\mathcal{F}}

\newcommand{\Q}{\mathcal{Q}}

\newcommand{\C}{\mathbb{C}}

\newcommand{\Pbb}{\mathbb{P}}

\newcommand{\Z}{\mathbb{Z}}

\newcommand{\N}{\mathbb{N}}

\newcommand{\tp}{\otimes}

\newcommand{\U}{U}

\newcommand{\ve}{\varepsilon}
\newcommand{\gm}{\gamma}

\newcommand{\dt}{\delta}

\newcommand{\op}{\oplus}
\newcommand{\la}{\lambda}
\newcommand{\tr}{\triangleright}
\newcommand{\tl}{\triangleleft}

\newcommand{\Char}{\mathrm{ch }}

\newcommand{\End}{\mathrm{End}}

\newcommand{\Span}{\mathrm{Span}}

\newcommand{\Hom}{\mathrm{Hom}}

\newcommand{\Rm}{\mathrm{R}}

\newcommand{\diag}{\mathrm{diag}}

\newcommand{\La}{\Lambda}

\newcommand{\g}{\mathfrak{g}}
\renewcommand{\b}{\mathfrak{b}}
\renewcommand{\k}{\mathfrak{k}}
\newcommand{\h}{\mathfrak{h}}

\newcommand{\n}{\mathfrak{n}}

\newcommand{\eps}{\epsilon}

\newcommand{\nn}{\nonumber}
\newcommand{\p}{\mathfrak{p}}
\renewcommand{\l}{\mathfrak{l}}

\newcommand{\si}{\sigma}
\newcommand{\al}{\alpha}
\newcommand{\bt}{\beta}

\newcommand{\be}{\begin{eqnarray}}
\newcommand{\ee}{\end{eqnarray}}

\newtheorem{thm}{Theorem}[section]
\newtheorem{propn}[thm]{Proposition}
\newtheorem{lemma}[thm]{Lemma}
\newtheorem{corollary}[thm]{Corollary}

\newcount\prg

\newcommand{\parag}{\advance\prg by1 {\noindent\bf\thesection.\the\prg\hspace{6pt}}}

\begin{document}

\title{Equivariant vector bundles over quantum projective spaces}
\author{
A. Mudrov\footnote{This study is supported  by the RFBR grant 15-01-03148.}
\vspace{20pt}\\
{\em \small  In memory of Ludwig Faddeev}
\vspace{10pt}\\
\small Department of Mathematics,\\ \small University of Leicester, \\
\small University Road,
LE1 7RH Leicester, UK\\
[0.1in]
}

\date{ }

\maketitle
\begin{abstract}
We construct
equivariant vector bundles over quantum  projective spaces making use of parabolic
Verma modules over the quantum general linear group.  Using an alternative realization of
the quantized coordinate ring of projective space  as a subalgebra in the algebra of functions on
the quantum group, we reformulate quantum vector bundles
in terms of quantum symmetric pairs. In this way, we prove complete
reducibility of modules over the corresponding coideal stabilizer subalgebras, via the quantum Frobenius reciprocity.
\end{abstract}
\\
{\small \underline{AMS classification codes}: 17B10, 17B37, 53D55.}
\newpage
\section{Introduction}
We quantize equivariant vector bundles over complex projective spaces regarded as one-sided projective modules over
coordinate rings. It is a natural next step in the
deformation quantization programme for Poisson manifolds, after quantization of their algebras of function, \cite{BFFLS},
considered as trivial bundles. In the presence of a symmetry group, quantization becomes a topic of representation theory.

The main idea underlying quantization is to realize the module of  global sections of an associated vector bundle by a vector space of 
linear maps between certain highest weight modules over the classical/quantum total group.
From this point of view the problem was addressed within the theory of dynamical twist in \cite{DM1}, for homogeneous spaces
with Levi isotropy subgroups. It was later realized that trivial vector bundles (e. g., algebras of functions) admit a similar approach 
with regard to 
semi-simple conjugacy classes (consisting of semi-simple elements) with non-Levi isotropy subgroups, \cite{M1,AM}. Those
findings suggest a  uniform quantization scheme for associated vector bundles over all semi-simple conjugacy classes,
making use of highest weight modules.

Technically quantization reduces to the following two questions.
Let $G$ be a reductive complex connected Lie group and $\g$ its Lie algebra. Consider a semisimple conjugacy class $O$ passing through a point $t$ in a fixed maximal torus $T$ and denote by $K\subset G$ the centralizer subgroup of $t$ with the Lie algebra $\k\subset \g$.
Let $U_q(\g)$ denote the standard quanization of the universal enveloping algebra $U(\g)$.  Associated with $t$ is a $U_q(\g)$-module $M$ of the highest weight  $\la=\la(t)$, such that the coordinate ring $\C[O]$ is quantized
as a subalgebra $\C_q[O]\subset \End(M)$.
Given a finite-dimensional irreducible $U_q(\g)$-module $V$, one can ask
\begin{enumerate}
  \item what are the highest weight submodules in $V\tp M$?
  \item  when does  $V\tp M$ split
in a direct sum of those submodules?
\end{enumerate}

The answer to the first question is easy for $K$ a Levi subgroup in $G$. Such are subgroups of maximal rank
whose basis of simple roots is a part of the total root basis of $G$. For such a subgroup, the universal enveloping
algebra is quantized as a Hopf subalgebra $U_q(\k)\subset U_q(\g)$, and  the submodules of interest are
parabolically induced from representations of $U_q(\k)$. Contrary to that, conjugacy classes with non-Levi $K$
have no natural quantum stabilizer in $U_q(\g)$, so  description of highest weight submodules 
in $V\tp M$ is not so straightforward.

The second question is challenging for both Levi and non-Levi $K$. It turns out that the
 answer can be formulated in terms of a contravariant form on
$V\tp M$,
\cite{M3}. The problem reduces to computing the determinant of a finite-dimensional matrix, which we call extremal twist.
In this paper, we apply the complete reducibility criterion of \cite{M3} to the simplest conjugacy classes
with Levi stabilizer subgroups.

We consider  complexification $\Pbb^n$ of the
projective space $\C P^n$ (regarded as a real affine variety). As a variety of rank-1 projectors on $\C^{n+1}$,  $\Pbb^n$ is isomorphic to a conjugacy class of the general linear
group $GL(n+1)$ with the stabilizer subgroup $GL(1)\times GL(n+1)$.
Specifically, different realizations of $\Pbb^n$ by conjugacy classes with the same stabilizer give rise to different Poisson structures on $\Pbb^n$.  They are induced by the Semenov-Tian-Shanski bracket on $GL(n+1)$, \cite{STS}. We show that almost all these Poisson structures admit quantization
of vector bundles on $\Pbb^n$. Moreover,  the quantization can be specialized to a particular value of the deformation  parameter $q\in \C$, provided it is not a root of unity and not an even root of $x_1x_2^{-1}$, where $x_1$ and $x_2$ are the two
different eigenvalues of the class representative $t$.

In accordance with the Serre-Swan theorem \cite{S,Sw}, 
a vector bundle $E$ over a base space $O$ translates to a sheaf of local sections equipped with the structure of a module over the sheaf of local functions on $O$.
It is, in fact, a bimodule, but the two-sided action is an incidental consequence of commutativity, which is  lost
in quantization. So by a vector bundle over a quantum space we understand a one-sided (left or right) projective module over its
non-commutative ``algebra of functions". Our examples of interest are affine varieties, in the classical limit, so we simplify  consideration
to the setting of global functions and global sections.

Under conditions to be  specified below (cf. Proposition \ref{proj_extr_tw}), $V\tp M$ splits into a direct sum of irreducible highest weight submodules
parameterized by irreducible $\k$-submodules in $V$.
They are separated by invariant idempotents, which  belong to $\End(V)\tp \C_q[\Pbb^n]$ and  give rise to projective $\C_q[\Pbb^n]$-modules identified with quantum vector bundles. In the classical limit, their fibers are the corresponding $\k$-submodules in $V$.

Since $\Pbb^n$ is a symmetric space, we  take advantage of an alternative presentation of the algebra $\C_q[\Pbb^n]$ via a quantum symmetric pair, \cite{Let,Kolb}.
There is a one-parameter family of solutions of the Reflection Equation associated with $U_q(\g)$, which
define
 one-dimensional representations of $\C_q[\Pbb^n]$ (regarded as ``points" on the quantum projective space). Every such a solution facilitates a realization
 of $\C_q[\Pbb^n]$ as
a subalgebra in the Hopf dual $\Tc$ to $U_q(\g)$, \cite{DM2}. At the same time, it defines a left coideal subalgebra $\Bc\subset U_q(\g)$
such that $\C_q[\Pbb^n]$ is identified with the subalgebra of $\Bc$-invariants in $\Tc$ under the right translation action.
The algebra $\Bc$ is a deformation of $U(\k')$, where $\k'\simeq \k$ is the isotropy Lie algebra  of the ``quantum
point" in the classical limit.

We prove that every finite-dimensional $U_q(\g)$-module $V$ is completely reducible over $\Bc$ and that the simple submodules are deformations
of irreducible $\k'$-submodules. They are readily constructed out of invariant idempotents from $\End(V)\tp \C_q[\Pbb^n]$
projecting $V\tp M$ onto the corresponding irreducible submodules.
This way we realize projective left $\C_q[\Pbb^n]$-modules by $\Bc$-invariants in $\Tc\tp X$, where
$X$ is a right $\Bc$-module and $\Tc$ is equipped with the right translation action. The invariants also carry a $\U_q(\g)$-action
via the left translations on $\Tc$. This is a deformation of a classical realization of the equivariant vector bundle
associated with a $\k'$-module $X$  (the fiber).

The set up of the paper is as follows. Section 2 contains basic facts about quantum groups and description
of the base module for  $\C_q[\Pbb^n]$.
At the beginning of Section 3, we summarize the results of \cite{M3} that are relevant to the present task.
We use them to study the structure of tensor products of finite-dimensional and the base module  thereafter.
In Section 4, we construct quantum vector bundles as projective (right) $\C_q[\Pbb^n]$-modules, based
on results of Section 3.
Then we switch to quantum symmetric pairs and give a realization of projective (left) $\C_q[\Pbb^n]$-modules
through $\Bc$-invariants, similarly to the classical induced representations.
\section{Preliminaries}
\label{SecPrelim}
Denote by  $\g$ the general linear algebra of complex $(n+1)\times (n+1)$-matrices.
Fix the Cartan subalgebra $\h\subset \g$ as the span of diagonal matrices and identify it with its dual vector space  $\h^*$  via the
trace inner product $(\>.\>,\>.\>)$. The root system $\Rm$ of $\g$ is expressed through the standard orthornomal basis
$\{\ve_i\}_{i=1}^n\subset \h^*$
as $\{\ve_i-\ve_j\}_{i\not =j}$ with the basis of simple positive roots $\Pi^+=\{\al_i\}_{i=1}^n$, where $\al_{i}=\ve_{i}-\ve_{i+1}$, $i=1,\ldots, n$. The subset $\Rm^+$ of all positive roots is $\{\ve_i-\ve_j\}_{i<j}$.

For all $\la\in \h^*$  denote by  $h_\la$ the element of $\h$ satisfying $\mu(h_\la)=(\mu,\la)$, for all $\mu\in \h^*$.
Denote by $\rho\in \h^*$ the half-sum of positive roots.

Suppose that $q\in \C$ is invertible and not a root of unity.
In what follows, we use the shortcuts $\bar q=q^{-1}$  and  $[x,y]_a=xy-ayx$ for $a\in \C$.

By $U_q(\g)$ we understand the standard quantum group \cite{Dr,ChP} generated over the complex field by $e_\al$, $f_\al$, and $q^{\pm h_\al}$,
$\al\in \Pi^+$, subject to
$$
q^{\pm h_\al}e_\bt=q^{\pm (\al,\bt)}e_\bt q^{\pm h_\al},
\quad
[e_\al,f_\bt]=\dt_{\al,\bt}[h_\al]_q,
\quad
q^{\pm h_\al}f_\bt=q^{\mp (\al,\bt)}f_\bt q^{\pm h_\al},\quad \al \in \Pi^+,
$$
where $[h_\al]_q=\frac{q^{h_\al}-q^{-h_\al}}{q-q^{-1}}$, and $q^{h_\al}q^{-h_\al}=1=q^{-h_\al}q^{h_\al}$.
The generators $e_{\al}$ and  $e_{-\al}=f_{\al}$ satisfy the  q-Serre relations
$$
[e_{\pm\al},[e_{\pm\al},e_{\pm\bt}]_q]_{\bar q}=0, \quad \mbox{if }\quad
 (\al,\bt)=-1, \quad\mbox{and }\quad [e_{\pm\al}, e_{\pm\bt}]=0\quad \mbox{if}\quad  (\al,\bt)=0.
$$

Fix the comultiplication on $U_q(\g)$ as
$$\Delta(f_\al)= f_\al\tp 1+q^{-h_\al}\tp f_\al,\quad \Delta(q^{\pm h_\al})=q^{\pm h_\al}\tp q^{\pm h_\al},\quad\Delta(e_\al)= e_\al\tp q^{h_\al}+1\tp e_\al.$$
Then the antipode acts on the generators by the assignment 
$$
\gm( f_\al)=- q^{h_\al}f_\al,\quad \gm( q^{\pm h_\al})=q^{\mp h_\al}, \quad \gm( e_\al)=- e_\al q^{-h_\al}
.
$$
 The counit homomorphism $\eps \colon U_q(\g)\to \C$
returns on the generators $\eps(e_\al)=\eps(f_\al)=0$, and $\eps(q^{\pm h_\al})=1$.

The assignment
$$\si\colon e_\al\to f_\al, \quad\si\colon f_\al\to e_\al, \quad \si\colon h_\al\to -h_\al$$
extends to an algebra involutive automorphism of $U_q(\g)$.
The involution $\omega=\gm^{-1}\circ \si$ preserves comultiplication and is an algebra anti-automorphism of $U_q(\g)$.

Denote by $U_q(\h)$,  $U_q(\g_+)$, $U_q(\g_-)$  the subalgebras in $U_q(\g)$  generated by $\{q^{\pm h_\al}\}_{\al\in \Pi^+}$, $\{e_\al\}_{\al\in \Pi^+}$, and $\{f_\al\}_{\al\in \Pi^+}$, respectively.
The quantum Borel subgroups  $U_q(\b_\pm)=U_q(\g_\pm)U_q(\h)$ are Hopf subalgebras in $U_q(\g)$.
Let $\k$ denote the subalgebra $\g\l(1)\op \g\l(n)\subset \g$ with the basis of simple roots $\Pi^+_\k=\{\al_2,\ldots,\al_n\}$.
The  elements $\{e_\al, f_\al, q^{\pm h_\al}\}_{\al \in \Pi^+_\k}$ generate a quantum subgroup
$U_q(\k)\subset U_q(\g)$.  A Hopf subalgebra $U_q(\p)=U_q(\k)U_q(\b_+)$ is a quantization of  $U(\p)$, where $\p =\k+\g_-$
is a parabolic Lie subalgebra. 

All $U_q(\g)$-modules under consideration are $U_q(\h)$-diagonalized. For any $U_q(\h)$-module
$V$ we denote by $V[\mu]\subset V$ the subspace of weight $\mu \in \h^*$, i. e., the
set of vectors $v\in V$ satisfying $q^{h_\al}v=q^{(\mu,\al)}v$ for all $\al\in \Pi^+$.
The set of weights of $V$ is denoted by $\La(V)$.

\section{Base module for quantum  projective space}
\subsection{Description of base module}
The complex projective space $\C P^n$  a real affine variety and can be realized as a set of  Hermitian projectors in  $\End(\C^{n+1})$
of complex rank 1.
It is a real form of the complex variety of projectors on $\C^{n+1}$ with unit trace.
The latter is a homogeneous $GL(n+1)$-variety that is isomorphic to a  conjugacy class passing through a matrix
$t=\diag(x_1,\ldots, x_n)\in GL(n+1)$ with $x_i=x_j$ for $i,j\geqslant 2$.
The centralizer of $t$ is the subgroup $GL(1)\times GL(n)$ whose Lie algebra is $\k$.

By quantization of the affine ring $\C[\Pbb^n]$ we call a $\C[q,q^{-1}]$-algebra $\C_q[\Pbb^n]$, which is a free module when 
extended over the local ring of rational functions in $q$ regular at $q=1$ and which is isomorphic to $\C[\Pbb^n]$ in the classical limit
$q\to 1$. We assume that $\C_q[\Pbb^n]$ is a $U_q(\g)$-module algebra, which structure is a deformation
of the classical structure of $U(\g)$-module algebra on $\C[\Pbb^n]$.

The quantization $\C_q[\Pbb^n]$ does exist and can be realized as a $U_q(\g)$-invariant subalgebra of of linear operators on
a highest weight module $M$, which we call {\em base module} for $\Pbb^n$, \cite{M1}.
Base modules form a 2-parameter family
delivering different quantizations of $\C[\Pbb^n]$. They are in bijection with diagonal matrices whose centralizer
Lie algebra is $\k$.


 The module $M$ has a PBW basis that makes it isomorphic to the vector space of polynomials in $n$ variables.
To describe this basis, let us introduce compound root vectors, cf. \cite{ChP}.
For $\al=\al_i+\ldots+\al_j\in \Rm^+$ with $i<j$ put
$$
e_{\al}=[e_{\al_j},\ldots [e_{\al_{i+1}},e_{\al_i}]_{q}\ldots]_{q}, \quad
f_{\al}=[\ldots[f_{\al_i}, f_{\al_{i+1}}]_{\bar q},\ldots f_{\al_{j}}]_{\bar q}
.
$$
We reserve the notation $\bt_i$ for the weights $\bt_i=\al_1+\ldots +\al_i$, $i=1,\ldots,n$.

Let $\la$ be a weight such that $q^{(\la,\ve_i)}=\sqrt{x_i}$ assuming the same root value for $i>1$. It
defines a 1-dimensional representation of the parabolic subalgebra $\U_q(\p)$ which is zero on $e_\al$ and $f_\al$ and returns
$q^{(\la,\ve_i)}$ on all $q^{h_{\ve_i}}$. Consider the parabolic $U_q(\p)$-module $M$ of highest weight $\la$ induced by the character 
of $U_q(\p)$ with canonical generator $1_\la\in M$. It is a quotient of the ordinary Verma module $\hat M_\la$ by the sum of submodules
generated by $f_\al 1_\la$, $\al \in \Pi^+_\k$.
The module $M$  has a basis of vectors $f_{\bt_1}^{m_1}\ldots f_{\bt_{n}}^{m_{n}}1_\la$, where $m_i$ take all
possible values in $\Z_+$.

Pick up  $s\in \C$ such that $q^{s}= q^{(\la,\al_1)}$ and put $|\vec m|_i=m_i+\ldots + m_n$ for all
$i=1,\ldots, n$.
Clearly $q^{2s}=x_1x_2^{-1}$.
\begin{lemma}
The action of  $U_q(\g)$ on $M$ is given by
\be
e_{\al_1} f_{\bt_1}^{m_1}\ldots f_{\bt_{n}}^{m_{n}}1_\la
&=& [m_1]_q[s-|\vec m|_1+1]_qf_{\bt_1}^{m_1-1}f_{\bt_2}^{m_2}\ldots f_{\bt_{n}}^{m_{n}}1_\la,
\nn\\
e_{\al_i} f_{\bt_1}^{m_1}\ldots f_{\bt_{n}}^{m_{n}}1_\la,
&=& -[m_i]_qq^{- m_i-1-s}f_{\bt_1}^{m_1}\ldots f_{\bt_{i-1}}^{m_{i-1}+1}f_{\bt_i}^{m_i-1} \ldots f_{\bt_{n}}^{m_{n}}1_\la,\quad i>1,
\nn\\
f_{\al_1} f_{\bt_1}^{m_1}\ldots f_{\bt_n}^{m_n}1_\la&=&f_{\bt_1}^{m_1+1} f_{\bt_{2}}^{m_{2}}
\ldots f_{\bt_n}^{m_n}1_\la,
\nn\\
f_{\al_i} f_{\bt_1}^{m_1}\ldots f_{\bt_n}^{m_n}1_\la&=&-q^{-1}[m_{i-1}]_qf_{\bt_1}^{m_1}\ldots f_{\bt_{i-1}}^{m_{i-1}-1} f_{\bt_{i}}^{m_{i}+1}
\ldots f_{\bt_n}^{m_n}1_\la,\quad i>1.
\nn
\ee
\label{action on M}
\end{lemma}
\vspace{-30pt}
\begin{proof}
  Direct caculation.
\end{proof}
\noindent
Recall that a weight vector in a $U_q(\g)$-module is called singular if it is killed by all $e_\al$, $\al \in \Pi^+$. Singular
vectors generate submodules of highest weight.
\begin{corollary}
   The module $M$ is irreducible if and only if $[s-m]_q\not =0$ for all $m\in \Z_+$.
\end{corollary}
\begin{proof}
 A singular vector in $M$ appears only when $q^{2s}\in q^{2\N}$ and it is proportional to $f_{\bt_1}^m1_\la$.
\end{proof}
\subsection{Canonical contravariant form on $V\tp M$ and extremal twist}
In this section we apply a criterion for a tensor product $V\tp M$ to be  completely reducible, following \cite{M3}.
Recall that a bilinear form $\langle \>.\>,.\>\rangle$ on a module $V$  is called contravariant with
respect to involution $\omega$ if
for all $v,w\in V$ and all $h\in U_q(\g)$ it satisfies $\langle h v, w\rangle=\langle v,\omega(h)w\rangle$.
It is known that every highest weight module has a unique, up to a scalar multiplier, contravariant form, 
which is non-degenerate if and only if the module is irreducible. It is called Shapovalov form of the highest
weight module.

Suppose that  $V$ is an irreducible  finite-dimensional $U_q(\g)$-module.
Denote by $(V\tp M)^+$ the subspace of singular vectors in $V\tp M$.
Introduce a canonical contravariant bilinear form on $V\tp M$ as the product of Shapovalov forms on $V$ and $M$.
Since $V$ and $M$ are irreducible,  the canonical form is non-degenerate on $V\tp Z$.
\begin{thm}[\cite{M3}]
  The tensor product $V\tp M$ is completely reducible if and only if the canonical form is non-degenerate when restricted
  to $(V\tp M)^+$.
\end{thm}
Next we describe a method of computing the determinant of the canonical form on $(V\tp M)^+$ that we use in what follows.
Let $\nu$ denote the highest (positive dominant) weight of $V$  and put $\ell_i=(\nu+\rho,\al_i)-1\in \Z_+$, $i=1,\ldots, n$.
There are two natural parameterizations of $(V\tp M)^+$ with subspaces in $V$ and $M$.
One of them is by the span $V^+_M$ of $U_q(\k)$-singular vectors,
i. e., those annihilated by $e_\al$, $\al \in \Pi_\k^+$ (Gelfant-Zeitlin reduction). 
Then  $\La(V^+_M)=\{\nu-\sum_{i=1}^{n}m_i\bt_i\}$ with $0\leqslant m_i\leqslant \ell_i$
for all $i$, and  all weights are multiplicity-free.

The alternative parametrization of $(V\tp M)^+$ is by a vector subspace $M^+_V\subset M$. 
Define a left ideal $I^+_V\subset U_q(\g_+)$ generated by $\{e_{\al_i}^{\ell_i+1}\}_{i=1}^n$.
Set $M^+_V\subset M$ as the nil-space of $I^+_V$.
The linear isomorphism between $M^+_V$ and $(V\tp M)^+$ acts by the assignment
$(V\tp M)^+\ni u=\sum_i v_i\tp w_i\mapsto \sum_i \langle v_i, 1_\nu\rangle w_i \in M^+_V$, where $1_\nu$ is the highest vector in $V$.

\begin{lemma}
The module $M$ splits into the orthogonal sum $M=M^+_V\op \omega(I^+_V)M$
with
$$
M^+_V =\Span\{ f_{\bt_1}^{m_1}\ldots f_{\bt_{n}}^{m_{n}}1_\la\}_{m_1\leqslant \ell_1, \ldots, m_n\leqslant \ell_n },\quad
\omega(I^+_V)M =\Span\{ f_{\bt_1}^{k_1}\ldots f_{\bt_{n}}^{k_{n}}1_\la\}_{k_1, \ldots, k_n},
$$
where  $k_i>\ell_i$ for some $i=1,\ldots,n$.
\end{lemma}
\begin{proof}
Observe that $\omega(I^+_V)M=\sum_{i=1}^{n}f_{\al_i}^{\ell_i+1}M$, which proves the right equality. The
left equality follows from Lemma \ref{action on M}.
This implies the  decomposition of $M$.
\end{proof}
Introduce a linear operator $\theta_{M,Z}\in \End(M^+_Z)$ as follows.
Pick up a singular vector  $u\in M\tp V\simeq V\tp M$ and 
present it as $u=\sum_i  w_i \tp f_i1_\nu = w\tp 1_\nu+\ldots $.
Put $\theta_{M,V}(w)$ equal to the image of $\sum_i \gm^{-1}(f_i)   w_i\in M$ under the orthogonal
projection to $M^+_V$.  This operator is independent of the chosen presentation of $u$.
We can regard $\theta_{M,V}$ as a bilinear form  $\langle \theta_{V,Z}(\>.\>),\>.\>\rangle\colon M^+_V\tp M^+_V\to \C$.
\begin{propn}[\cite{M3}]
\label{V-Z-extr}
  The form $\langle \theta_{M,V}(\>.\>),\>.\>\rangle$ is the pullback of the canonical form restrected to $(M\tp V)^+$,
  under the isomorphism $M^+_V\to (M\tp V)^+$.
\end{propn}
Under certain conditions, see \cite{M3}, Proposition 7.1, the operator $\theta_{M,V}$ can be obtained from
a lift $\Fc\in U_q(\g_+)\hat \tp U_q(\g_-)$ of the inverse invariant pairing between an irreducible Verma module
$\hat M_\mu$ (of highest weight $\mu$) and its right dual ${}^*\!\hat M_\mu$ (the opposite Verma module of lowest weight $-\mu$).
Such a lift is the image of $1\in \C$ under the composition of maps
$$\C\to {}^*\!\hat M_\mu\hat \tp \hat M_\mu \to U_q(\g_+)\hat \tp U_q(\g_-)\quad \mbox{(a completed tensor products)},$$
where the right arrow is a $\C$-linear section of the natural homomorphism of $U_q(\g_+) \tp U_q(\g_-)$-modules
(the left arrow is the coevaluation map). 
Using the Sweedler notation  $\Fc_1\tp \Fc_2$  for $\Fc$ we set $\Upsilon_\mu=\gm^{-1}(\Fc_2)\Fc_1\in U_q(\g)$ (a certain extension
of the quantum group assumed).
The module $V$ is a quotient of $\hat M_\nu$, therefore the lift $\Fc$ has poles at $\mu=\nu$. However, its part that is relevant to $M^+_V$ is well-defined if the following hypothesis holds true.
\begin{propn}[\cite{M3}]
\label{lift_of_twist}
  Suppose that every singular vector in $V\tp M$ is the image of a singular vector in $\hat M_\mu\tp M$
  under the natural homomorphism $\hat M_\mu \tp M\to V\tp M$.
  Then $\langle \theta_{M,V}(\>.\>)\>, \>.\> \rangle =
\lim_{\mu\to \nu}\langle \Upsilon_\mu(\>.\>)\>, \>.\> \rangle$ on $M^+_V$.
\end{propn}
\subsection{Complete reduciblility of $V\tp M$}
The inverse of the operator $\Upsilon_\mu$ is found in \cite{EV}, in connection with the dynamical
Weyl group.
Their result can be used for calculation of $\theta_{M,V}$ because the conditions of Proposition \ref{lift_of_twist}
are fulfilled, as proved below.
\begin{propn}
\label{sing_cover}
  All singular vectors in $V\tp M$ are images of singular vectors in $\hat M_\nu \tp M$.
\end{propn}
\begin{proof}
 Thanks to the presence of the Gelfand-Zeitlin basis in $V$ (see e. g., \cite{Mol} for the classical group version),
  every element of $v\in V^+_M$ can be presented as $v=f(\nu)1_\nu$, where $f(\nu)\in U_q(\g_-)$ belongs to the normalizer
  of the left ideal $U_q(\g)I^+_M$ in $U_q(\g)$ (the quotient of the normalizer by $U_q(\g)I^+_M$ is the quantum Mickelsson algebra of the pair $(\g,\k)$, cf. \cite{KO}). Therefore $v$ is the image of
  an element $\hat v\in (\hat M_\nu)^+_M\simeq (\hat M_\nu\tp M)^+$. The corresponding singular vectors
  in $\hat M_\nu\tp M$ and $V\tp M$ are, respectively,  $\hat u=\Fc_M(\hat v\tp 1_\la)$ and $ u=\Fc_M(v\tp 1_\la)$,
  where $\Fc_M\in U_q(\g_+)\hat \tp U_q(\g_-)$ is a lift of the inverse invariant pairing between $M$ and ${}^*\!M$. 
  Then $u$ is the image of $\hat u$.
\end{proof}

Since the weights in $M$ are multiplicity free, we can write, following \cite{EV},
$$\Upsilon_\nu( w)\propto \prod_{\al\in \Rm_\g^+}\prod_{k=1}^{l_{\xi,\al}}\frac{[(\nu+\rho+\xi,\al)+k]_{q}}{[(\nu+\rho,\al)-k]_{q}}w,
\quad w\in M^+_V[\xi],
$$
where $l_{\xi,\al}=\max\{l \in \Z:e_\al^lw\not =0 \}$.
The operator $\Upsilon_{\nu}$ preserves $M^+_V$, thus $\theta_{M,V}=\Upsilon_\nu$ on $M^+_V$.

\begin{propn}
 The operators $\theta_{M,V}$ are invertible  for all $V$ {\em iff} $q^{2s}\not \in  q^{2\Z}$.
 \label{proj_extr_tw}
\end{propn}
\begin{proof}
Let $\xi$ be a weight from $\La(M^+_V)$, that is, $\xi=\la-\sum_{i=1}^{n}m_i\bt_i$ with $0\leqslant m_i\leqslant \ell_i=(\nu+\rho,\al_i)-1$.
It is easy to check, using Lemma \ref{action on M}, that for  $\al=\ve_i-\ve_{j+1}$, $i\leqslant j+1$,
the integer $l_{\xi,\al}$ equals $m_j$.
Then the factor $[(\nu+\rho,\al)-k]_{q}$  equals $[m_i+\ldots+m_j+(j-i+1)-k]_{q}\not =0$, as $q$ is not a root of unity. Therefore the denominator of $\theta_{M,V}$ does not vanish.

Define $\phi_{\xi,\al,k}=[(\nu+\rho+\xi,\al)+k]_{q}$ for all $\xi\in \La(M^+_V)$, $\al\in \Rm^+$, and $k\in \N$.
Then $\det \theta_{M,V}$ is equal to the product
$
\prod_{\xi\in \La(M^+_V)}\prod_{\al\in \Rm_\g^+}\prod_{k=1}^{l_{\xi,\al}}\phi_{\xi,\al,k}
$,
up to a non-zero factor.

If $\al\in \Pi^+_\k$, then $(\nu+\rho+\xi,\al)+k\in \Z\backslash \{0\}$, and $\phi_{\xi,\al,k}\not =0$  since $q$ is not a root of unity.
Suppose that $\al =\bt_j$. Then $\phi_{\xi,\bt_j,k}=[(\nu+\rho,\bt_j)-\sum_{i=1}^{n}m_i(\bt_i,\bt_j)+k+s]_q$.
The first three terms are obviously integer, therefore $\det \theta_{M,V}\not =0$ if $s$ satisfies the hypothesis.

Conversely,  suppose that $[s-m]_q=0$ for some $m\in \Z$ and show that $\det \theta_{M,V}=0$ if $\ell_1$ is
sufficiently large.  Put $m_i=0$ for $i>1$, then $1+\ell_1-2m_1 +k$ with $m_1\in [0,\ell_1]$ and $k\in [1,m_1]$ takes all integer values in
the interval $[2-\ell_1,\ell_1]$. Therefore, one of the factors $\phi_{\xi,\bt_1,k}$ vanishes once $|m-1|\leqslant \ell_1-1$.
\end{proof}
\begin{corollary}
Under the assumptions of Proposition \ref{proj_extr_tw}, the modules $V\tp M$ are completely reducible for all finite-dimensional
$V$.
\end{corollary}

\section{Equivariant  vector bundles over quantum projective spaces}
\subsection{Parabolic Verma modules}
The base module $M$ is an example of parabolic induction from a one-dimensional representation
of $U_q(\k)$. Let us remind this construction in general.
The  Lie subalgebras $\p_\pm=\k+\g_\pm\subset \g$ are called parabolic with the Levi factor $\k$.
Their nilradicals $\n_\pm$ are spanned by the vectors of roots from $\Rm_\g^\pm\backslash \Rm_\k^\pm=\{\pm\bt_i\}_{i=1}^n$ respectively, 
thus $\p_\pm=\k\ltimes \n_\pm$.
Denote by $U_q(\n_\pm)$ the subalgebra in $U_q(\g)$ generated by $\{e_{\bt_i}\}_{i=1}^n$, and similarly
by $U_q(\n_-)$ the subalgebra  generated by $\{f_{\bt_i}\}_{i=1}^n$.
The algebra $U_q(\g)$ enjoys a factorization $U_q(\g)=U_q(\n_\mp)U_q(\p_\pm)$ with $U_q(\p_\pm)=U_q(\k)U_q(\n_\pm)$, facilitated by the PBW basis.

Fix a finite-dimensional irreducible $U_q(\k)$-module $X$ of highest weight $\xi$ and  make it a module over $U_q(\p_+)$ assuming the trivial action of $U_q(\n_+)$.
Denote by $M_X$ the induced module $U_q(\g)\tp_{U_q(\p_+)}X$.
One similarly defines  parabolic Verma modules of lowest weight, $N_{X}=U_q(\g)\tp_{U_q(\p_-)}X$.
Applying Frobenius reciprocity twice, one can check that  $M_{X}\tp N_{Y}$ is isomorphic to the $U_q(\g)$-module
 induced
from the $U_q(\k)$-module $X\tp Y$.

Recall that the character $\Char(W)$ of a $U_q(\h)$-module $W$ is defined as the formal sum
$\sum_{\mu\in \La(W)}\dim W[\mu] q^{\mu}$, where $q^\mu$ is a homomorphism $U_q(\h)\to \C$ acting by $q^{h_\al}=q^{(\mu,\al)}$ for all $\al \in \Pi^+$. For a pair of $U_q(\h)$-modules we write $\Char(W_1)\leqslant \Char(W_2)$ if $\dim W_1[\mu]\leqslant \dim W_2[\mu]$ for all $\mu$.
The triangular factorization of $U_q(\g)$ implies that $M_X\simeq U_q(\n_-)X$.
Therefore the character of $M_X$ is equal  to $\Char(M_\C)(\Char X)$, where $M_\C$ is parabolically induced from the trivial $U_q(\k)$-module.

Given a weight $\la$ subject to $q^{2(\la,\al)}=1 $ for all $\al \in \Pi^+_\k$, one can define a
one-dimensional  $U_q(\k)$-module $\C_\la$ where $U_q(\h)$ acts by $q^{h_\al}\mapsto q^{(\la,\al)}$.
Then for any finite-dimensional $U_q(\k)$-module $X$ we set $X_\la=X\tp \C_\la$. We fix $\la$ to be the highest weight of
the base module $M$.
\begin{thm}
  Let $V$ be a finite-dimensional $U_q(\g)$-module and $V=\op_i X^i$ its irreducible decomposition over $U_q(\k)$.
  Suppose that $q^{2(\la,\al_1)}\not \in q^{2\Z}$.
  Then $V\tp M\simeq \op_i M_{X^i_\la}$ is an irreducible decomposition.
\end{thm}
\begin{proof}
  Since $M$ is irreducible, $\Hom(M_{X^i_\la},V\tp M)\simeq \Hom (M_{X^i_\la}\tp N,V)\simeq \Hom_{U_q(\k)}(X^i,V)$.
  We will prove that all homomorphisms $M_{X^i_\la}\to V\tp M$ are embeddings. Let $M_i$ denote the image of $M_{X^i_\la}$.
  By Proposition \ref{proj_extr_tw}, $V\tp M=\op_i M_i$ is an irreducible decomposition. Therefore $\sum_{i}\Char(M_i)=\Char(M)\Char(V)$.
  On the other hand,
  $\sum_i\Char(M_{X^i_\la})=\sum_i\Char (M)\Char(X^i)=\Char (M)\Char(V)$.
  Since $\Char(M_{i})\leqslant \Char(M_{X^i_\la})$, that is possible only if $\Char(M_{i})= \Char(M_{X^i_\la})$ and therefore
   $M_i\simeq M_{X^i_\la}$ for all $i$.
\end{proof}
\noindent
Note that one cannot relax the condition on the weight $\la$ as $V$ varies over all finite-dimensional modules.

Remark that $x_i=q^{2(\la,\ve_i)}$ should be treated as constants rather than functions of $q$
(e.g., via  rescaling of  $\la$).
They have the meaning of eigenvalues of matrices comprising the conjugacy class 
that represents $\Pbb^n$. Once $x_1 x_2^{-1}\not \in q^{2\Z}$,
the modules $M_{X_\la}$ can be extended over the ring of formal power series $\C[[\hbar]]$ via $q=e^\hbar$.
This trick will be used in the next section in the quantization context.

\subsection{Projective modules over $\C_q[\Pbb^n]$}
By a classical equivariant vector bundle over $\Pbb^n$ with fiber $X$ we understand the (left or right) projective $\C[\Pbb^n]$-module of global sections denoted by $\Gamma(\Pbb^n,X)$. It can be  realized as the subspace of $\k$-invariants in
$\C[G]\tp X$, $G=GL(n+1)$, where $\k$ acts on $\C[G]$ by translations.

We understand  quantization as deformation, when complex vector spaces are extended to free $\C[[\hbar]]$-modules.
In our situation, when all modules and maps are rational in $q$, they can be considered over the
local ring of rational functions in $q$ regular at $q=1$. They can be specialized at generic $q$ meaning that
$q$ can can take all but a finite set of complex numbers including $q=1$, in every submodule of finite rank.

Let $\Ac\subset \End(M)$ denote the quantized polynomial algebra $\C_q[\Pbb^{n}]$.
By quantization of an equivariant vector bundle on $\Pbb^{n}$ in this subsection we mean a $U_q(\g)$-equivariant deformation of
$\Gamma(\Pbb^n,X)$ in the class of right $\C_q[\Pbb^{n}]$-modules. We will realize it as $\hat P (V\tp \Ac)$,
where $V\supset X$ is a finite-dimensional $U_q(\g)$-module and $\hat P\in \End(V)\tp \Ac$ is an $\U_q(\g)$-invariant idempotent.
Such idempotents can be constructed via a  direct sum decomposition of $V\tp M$ due to the following fact.

\begin{propn}
For all $q$ except for a finite set, every invariant projector from $V\tp M$ onto an irreducible submodule belongs to  $\End(V)\tp \Ac$.
\end{propn}
\begin{proof}
For generic $q$, the algebra  $\Ac$ is embedded in the locally finite part $\End^\circ(M)$ of  the adjoint $U_q(\g)$-module $\End(M)$, which  is isomorphic to the locally finite part of $M\tp N$.  Frobenius reciprocity yields
  $$\Hom_{U_q(\g)}(M\tp N,V)\simeq\Hom_{U_q(\g)}(M,V\tp M)\simeq \Hom_{\k}(\C,V)\simeq \Hom_{U_q(\g)}(\Ac,V)$$
for every finite-dimensional $U_q(\g)$-module $V$.
Therefore the image of $\Ac$ in $\End(M)$ coincides with $\End^\circ(M)$
(this is a quantum group version of Kostant's  problem,  cf. \cite{KST}).
 An invariant projector $\hat P\in \End(V\tp M)$ is a matrix with entries in $\End^\circ(M)$, therefore it is in $\End(V)\tp \Ac$.
 Since $\hat P$ is rational in $q$ and all isotypic components of $\Ac$  are finite-dimensional, this is true for all  but a finite
 number of values of $q$.
\end{proof}
\noindent

Irreducible decomposition of $V\tp M$  depends on a choice of basis in $V^{\k_+}=V^+_M$.
Let $\hat P\in \End(V)\tp \Ac$ be the projector to a particular copy of $M_{X_\la}$ determined by $v\in V^{\k_+}\cap X$.
Here $v$ is the highest vector of an irreducible $\k$-submodule $X$.
\begin{thm}
\label{hom-bund1}
The right $\Ac$-module $\hat P(V\tp \Ac)$  is a quantization of $\Gamma(\Pbb^n,X)$.
\end{thm}
\begin{proof}
  The  left $U_q(\g)$-equivariant  right $\Ac$-module $\hat P(V\tp \Ac)$ can be realized
  on the locally finite part of $\Hom_\C(M, M_{X_\la})$ and is isomorphic to  locally finite part of $M_{X_\la}\tp N$ as a $U_q(\g)$-module,
  for generic $q$.
  The latter is proved to be isomorphic to subspace of $\k$-invariants in $\C[G]\tp X$, thanks to  the Peter-Weyl decomposition
  $\U^*_q(\g)=\sum_{irrep\> V}V\tp V^*$ (summation taken is over classes of irreducible finite-dimensional modules). So, $X$ is the fiber of the bundle in the classical limit.
\end{proof}


\subsection{Coideal stabilizer subalgebra}

We switch to an alternative realization of quantum vector bundles in terms of quantum symmetric pairs.
It is convenient to pass to the left $\Ac$-module version.

Let $\Tc$ denote the Hopf dual of $U_q(\g)$ and $T=(T_{ij})_{i,j=1}^{n+1}$ its matrix of generators.
It is invertible with $(T^{-1})_{ij}=\gm(T_{ij})$, where
$\gm$ is the antipode of $\Tc$.
One has two commuting left and right translation actions of $U_q(\g)$ on $\Tc$ expressed through the Hopf paring and the
comultiplication in $\Tc$ by
$$
h\tr a= a^{(1)}(h,a^{(2)}),\quad  a\tl h= (a^{(1)},h) a^{(2)},\quad  a\in \Tc, \quad h\in U_q(\g).
$$
They are compatible with multiplication on $\Tc$ making it a $U_q(\g)$-bimodule algebra.

Denote by $\pi\colon U_q(\g)\to \End(\C^{n+1})$ the natural representation homomorphism.
Let $\Ru$ be a universal R-matrix of $U_q(\g)$.
The element  $\Ru_{21}\Ru_{}$  commutes with the coproduct $\Delta(x)$ for all $x\in U_q(\g)$.
Introduce the matrix $\Q=(\pi\tp \id)(\Ru_{21}\Ru_{})$ with entries in $U_q(\g)$.
Denote by $R=(\pi\tp \pi)(\Ru) \in \End(\C^{2n+1})\tp \End(\C^{2n+1})$ the image of $\Ru$.
One can choose  $\Ru$ such that $R$ is proportional to the $R$-matrix used in
\cite{FRT}. Following   \cite{M2}, introduce a matrix $A\in \End(\C^{2n+1})$ by
$$
A=\quad
\left(\begin{array}{cccccc}
  x_1+q^{-2}x_2& 0 & \ldots & 0 &c \\
  0 &q^{-2}x_2& 0 & 0 &0 \\
  \vdots & 0 &\ddots&0&\vdots\\
  0 & 0 & 0 &q^{-2}x_2&0  \\
  d  & 0 & \ldots &0&0  \\
\end{array}
\right),
$$
where $cd=-q^{-2}x_1x_2\in \C\backslash\{0\}$. It
solves the Reflection Equation  \cite{KS}
$$
R_{21}A_1R_{12}A_2=A_2R_{21}A_1R_{12}\in \End(\C^{2n+1})\tp \End(\C^{2n+1}),
$$
where the indices mark the tensor factors.
The assignment $\Q\mapsto T^{-1}A T$
defines an equivariant embedding  $\Ac\subset  \Tc$, where $\Tc$
is regarded as a $U_q(\g)$-module  under the left translation action. The matrix $A$ determines a homomorphism
$\chi\colon \Ac\to \C$, $\Q_{ij}\mapsto A_{ij}$. It factors through the composition $\Ac\to  \Tc\to \C$,
where the right arrow is the counit $\epsilon$.

The entries of the matrix $(\pi\tp \id)(\Ru_{21})A_1(\pi\tp \id)(\Ru_{12})\in \End(\C^{2n+1})\tp \U_q(\g)$ generate a left coideal subalgebra $\Bc\subset \U_q(\g)$,
such that $\Ac$ is the subalgebra of $\Bc$-invariants: $a\tl b=\epsilon(b)a$ for all $b\in \Bc$ and $a\in \Ac$. 
The algebra $\Bc$ is a deformation of $U(\k')$,
where $\k'\simeq \k$.

Let $V$ be a finite-dimensional $U_q(\g)$-module and $\hat P\in V\tp \Ac$ an invariant idempotent. In the Sweedler notation, 
the invariance reads as $\hat P_1\tp h\tr\hat P_2=\gm(h^{(1)})\hat P_1h^{(2)}\tp \hat P_2$ for all $h\in U_q(\g)$.
\begin{lemma}
  The  projector  $P=\hat P_1\chi(\hat P_2)\in \End(V)$ commutes with $\Bc$.
\end{lemma}
\begin{proof}
Recall that $\chi$ is the restricted counit $\epsilon$ once $\Ac$ is realized as a subalgebra in $\Tc$.
For all $b\in \Bc$ and all $a\in \Ac$ one has
$\epsilon(b\tr a)=(b,a)=\ve(a\tl b)=\epsilon(b)\epsilon(a)$. Then
 $\gm(b^{(1)})P b^{(2)}=\gm(\hat P_1)\epsilon(b\tr\hat P_2)=\epsilon(b)P$, which implies the lemma.
\end{proof}
\begin{propn}
\begin{enumerate}
  \item Every finite-dimensional (right) $U_q(\g)$-module $V$ is completely reducible over $\Bc$ for all $q$ except for a finite set.
  \item Each irreducible $\Bc$-submodule in $V$ is a deformation of a classical $U(\k')$-submodule with the same multiplicity.
  \item Each $\Bc$-submodule in $V$ is the image of a $\Bc$-invariant projector $(\id\tp \chi)(\hat P)$, where
  $\hat P\in End(V)\tp \Ac$ is  a $U_q(\g)$-invariant idempotent.
\end{enumerate}
\end{propn}
\begin{proof}
Consider the unit resolution   $\id_{V\tp M} = \sum_{i}\hat P_i$  by simple invariant idempotents  $\hat P_i\in \End(V\tp \Ac)$, for generic $q$.
Let $P_i=(\id_V\tp \chi)(\hat P_i)$ be a projector to a $\Bc$-submodule in $V$.
All $P_i$ are rational functions in $q$ and deliver the unit resolution
$\id_V = \sum_{i}P_i$ for all  $q$ but from a finite set. This implies that the ranks of $P_i$ stay constant where they are defined.
In the classical limit, $X_i=VP_i$ is irreducible over $\k'$, therefore  $\Bc$ fills up entire $\End(X_i)$. That is so for almost all $q$,
since image cannot reduce in deformation. Therefore $X_i$ is irreducible over $\Bc$
for all but a finite number of $q$. This proves the first and second statement and readily implies the third since every irreducible
submodule is then separable as a direct summand. The corresponding $U_q(\g)$-invariant idempotent
can be recovered as shown in Lemma \ref{2projectors} below.
\end{proof}

\subsection{Quantum vector bundles via symmetric pairs}
\label{SecSymPair}

In this section we realize associated vector bundles
 as $\Bc$-invariants in the tensor product of $\Tc$ and right $\Bc$-modules. We will work with right modules, which corresponds to the right coset picture.


Thanks to the Peter-Weyl decomposition of $\Tc$, every finite-dimensional right $U_q(\g)$-module $V$ is a left $\Tc$-comodule.
We use a Sweedler-like notation for the left coaction $\delta\colon V\to \Tc\tp V$, namely $\delta\colon v\mapsto v^{(1)}\tp v^{[2]}$.
Then $v\tl h=(v^{(1)},h)v^{[2]}$ for $v\in V$, $h\in U_q(\g)$.

For each $\Tc$-comodule $V$, define two $\Tc$-linear automorphisms of the  left $\Tc$-module $\Tc\tp V$,
$$
\iota\colon a\tp v\mapsto av^{(1)}\tp v^{[2]}, \quad \bar \iota \colon a\tp v \mapsto  a\gm(v^{(1)})\tp v^{[2]}.
$$
It is easy  to check that $\iota\bar \iota=\bar \iota \iota=\id$ and
\be
\label{trivials}
\iota (\Tc\tp V)^\Bc=\Ac\tp V, \quad \bar \iota(\Ac\tp V)=(\Tc\tp V)^\Bc.
\ee
We make $V$ a left $U_q(\g)$-module via the action $h\btr v=v\tl \gm(h)$, $v\in V$,
and consider $\Ac\tp V$ as a left  $U_q(\g)$-module. The tensor product $\Tc\tp V$ is
also a left $U_q(\g)$-module with respect to the left translations on $\Tc$ and the trivial
action on $V$. Then the map $\iota$ is $U_q(\g)$-invariant, as well as the subspace $\iota (\Tc\tp V)^\Bc$.

The algebra $\End(V)$ becomes a natural right comodule whose
coaction satisfies the identity $(E^{(1)},h) E^{[2]}=\gm(h^{(1)})Eh^{(2)}$ for all $E\in \End(V)$ for all $h\in U_q(\g)$.
Let $\hat P\in \End(V)\tp \Ac$ be an invariant indecomposable idempotent. If $\{E_i\}\subset \End(V)$ is a basis, then $\hat P=\sum_{i} E_i\tp a_i$ for some $a_i\in \Ac\subset \Tc$. The $\Bc$-invariant projector $P=\sum_{i}E_i\eps(a_i)$ defines
an irreducible right $\Bc$-submodule $X=VP$.

\begin{lemma}
\label{2projectors}
  One has
$ \iota(1\tp  P)=\hat P_{21}$ and $\bar \iota (\hat P_{21})=1\tp P.$
\end{lemma}
\begin{proof}
  Indeed, the $U_q(\g)$-invariance of $\hat P$ implies $\sum_{i} h \tr a_i \tp E_i=\sum_{i}a_i \tp E_i\tl h$ for all $h\in U_q(\g)$.
  Pairing the $\Tc$-factor of $\bar \iota (\hat P_{21})$ with arbitrary $h\in U_q(\g)$ we get
  $$
  \sum_{i} \bigl(a_i\gm(E_i^{(1)}),h\bigr)\tp  E_i^{[2]}=\sum_{i} (a_i,h^{(1)})\bigl(\gm(E_i^{(1)}),h^{(2)}\bigl)\tp E_i^{[2]}
  =
  \sum_{i} (a_i,h^{(1)})\tp E_i\tl \gm(h^{(2)}).
  $$
The last term is $$ \sum_{i} \bigl(\gm(h^{(2)})\tr a_i,h^{(1)}\bigr)\tp  E_i=\sum_{i} \bigl(a_i,h^{(1)}\gm(h^{(2)})\bigr)\tp E_i=\sum_{i} \epsilon(h)\eps(a_i)\tp  E_i,$$ due to invariance of $\hat P$.
This proves the right equality, which implies the left one.
\end{proof}
The subspaces $(\Ac\tp V)\hat P$ and $(\Tc\tp VP)^\Bc$ in $\Tc\tp V$ are invariant under the left action of $\Ac$
restricted from $\Tc$. They are also $U_q(\g)$-modules as explained above.
\begin{propn}
  The map $\bar \iota\colon (\Ac\tp V)\hat P_{21}\to (\Tc\tp VP)^\Bc$ establishes an isomorphism of the left $\Ac$-modules
  and $U_q(\g)$-modules.
\end{propn}
\begin{proof}
Indeed, pick up $v\in V$. Then by Lemma \ref{2projectors}, we get
$$\bar \iota \bigl((1\tp v)\hat P_{21}\bigr)= \sum_{i} a_i\gm(E_i^{(1)})\gm(v^{(1)})\tp  v^{[2]}E_i^{[2]}=\bar \iota (1\tp v)(1\tp P).$$
The right formula in (\ref{trivials}) and $\Bc$-invariance of $P$ yield $\bar \iota (1\tp v)P\in (\Tc\tp VP)^\Bc$, 
again by Lemma \ref{2projectors}.  
On the other hand, 
$$
\iota (1\tp vP)=\iota(1\tp v)\iota(1\tp P)=\iota (1\tp v)\hat P_{21}\in (\Ac\tp V)\hat P_{21}.
$$
 Now $U_q(\g)$- and $\Ac$-linearity of $\iota$ and $\bar\iota$ complete the proof.
\end{proof}
 Fix  an invariant idempotent $\hat P$ with the corresponding $\Bc$-invariant projector $P$ and consider the right $\Bc$-module  $X=VP$.
\begin{thm}
The $\Ac$-module $(\Tc\tp X)^\Bc$ is a quantization of the vector bundle $\Gamma(\Pbb^n,X)$.
\end{thm}
\begin{proof}
The map $\iota \colon (\Tc\tp V)^\Bc\to \Ac\tp V$ is an isomorphism of $U_q(\g)$-modules.
Let $\hat P\in \End(V)\tp \Ac$ be an invariant projector.
The left $U_q(\g)$-module $(\Ac\tp V)\hat P_{21}$ is isomorphic to $\Hom_\C(\!M_{X_\la},M)\simeq \Hom_{U_q(\k)}(\Tc,X)$. Due to the Peter-Weyl decomposition, the latter
is a quantization of $(\C[G]\tp X)^\k$ as required.
\end{proof}

\end{document}